%% file: main.tex
\documentclass[review,3p]{elsarticle}

\usepackage{framed,multirow}

\usepackage{amssymb}
\usepackage{latexsym}

\usepackage{url}
\usepackage{xcolor}
\definecolor{newcolor}{rgb}{.8,.349,.1}
\usepackage{bm} 
\usepackage{physics}
\usepackage{amsmath}
\usepackage{amsthm}
\usepackage{mathtools}
\usepackage{booktabs}
\usepackage{yhmath}
\newcommand{\volume}{{\ooalign{\hfil$V$\hfil\cr\kern0.08em--\hfil\cr}}}
\usepackage{graphicx}
\usepackage[labelfont=it]{caption}
\usepackage{subcaption}
\usepackage{float}

\usepackage{cancel}
\usepackage[normalem]{ulem}

\usepackage{algorithm}
\usepackage{algpseudocode}
\usepackage{multicol}
\usepackage{tikz}

\usepackage{pifont}
\usepackage{hyperref}

\usepackage{lineno}
\newcommand\patchAmsMathEnvironmentForLineno[1]{%
\expandafter\let\csname old#1\expandafter\endcsname\csname #1\endcsname
\expandafter\let\csname oldend#1\expandafter\endcsname\csname end#1\endcsname
\renewenvironment{#1}%
{\linenomath\csname old#1\endcsname}%
{\csname oldend#1\endcsname\endlinenomath}}%
\newcommand\patchBothAmsMathEnvironmentsForLineno[1]{%
\patchAmsMathEnvironmentForLineno{#1}%
\patchAmsMathEnvironmentForLineno{#1*}}%
\AtBeginDocument{%
\patchBothAmsMathEnvironmentsForLineno{equation}%
\patchBothAmsMathEnvironmentsForLineno{align}%
\patchBothAmsMathEnvironmentsForLineno{flalign}%
\patchBothAmsMathEnvironmentsForLineno{alignat}%
\patchBothAmsMathEnvironmentsForLineno{gather}%
\patchBothAmsMathEnvironmentsForLineno{multline}%
}
\theoremstyle{remark} 
\newtheorem{remark}{Remark}

\theoremstyle{plain}

\theoremstyle{plain} 
\newtheorem{theorem}{Theorem}

\theoremstyle{definition}
\newtheorem{definition}{Definition}

\journal{Journal of Computational Physics}

\begin{document}


\begin{frontmatter}

\title{An Adaptive Flux Reconstruction Scheme for Robust Shock
Capturing}%
\author[1]{Sai Shruthi Srinivasan\corref{cor1}}
\cortext[cor1]{Corresponding author.}

\ead{sai.srinivasan@mail.mcgill.ca}
\author[1]{Siva Nadarajah}
\address[1]{Department of Mechanical Engineering, McGill University, Montreal, Quebec H3A OC3, Canada}




\end{frontmatter}
\vspace{-1.5cm}
\input{sections/afr}
\input{sections/methodology}
\input{sections/results.tex}
\input{sections/conclusion.tex}
\input{sections/acknowledgements}
\bibliographystyle{model1-num-names}
\bibliography{refs}

\end{document}

%% file: sections/afr.tex
\section{Introduction}\label{sec: intro}

In the case of hyperbolic conservation laws, high-order methods, such as the classical DG method, experience the phenomenon of unwanted high-frequency oscillations in the vicinity of a shock. Shock-capturing methods such as artificial dissipation, solution, flux, or TVD limiting are generally used to eliminate nonphysical oscillations and provide bounds on physical quantities.
For entropy-stable schemes, the additional objective would be to retain provable entropy dissipation guarantees of the underlying scheme, i.e. subcell limiting or entropy filtering~\cite{vilar2019posteriori,hennemann2021provably,rueda2021subcell,dzanic2022positivity}. 
The nonlinearly-stable flux reconstruction (NSFR) semi-discretization given in Eq.~\ref{eq: nsfrStrong} with a suitable flux reconstruction scheme has been demonstrated to mitigate spurious oscillations in the presence of shock discontinuities and at CFL values substantially larger than the DG variant of the NSFR scheme whilst retaining the property of entropy stability \cite{srinivasan2025investigation}.
NSFR schemes achieve this by introducing an alternative lifting operator for surface numerical flux penalization, albeit at the expense of accuracy. In this technical note, we present an adaptive approach to the choice of the lifting operator employed, which maintains higher accuracy and allows for larger CFL values while retaining the underlying provable attributes of the scheme.
While it cannot \textit{eliminate} oscillations such as the aforementioned shock-capturing methods, together with a positivity-preserving limiter, the scheme provides for solutions that are essentially oscillation-free.  
\begin{equation}\label{eq: nsfrStrong}
    (\bm{M}_m+\bm{K}_m)\frac{d}{dt}\hat{\bm{u}}_m(t)^T + [\bm{\chi}(\bm{\xi_v^r})^T\bm{\chi}(\bm{\xi_f^r})^T]\left[\left(\tilde{\bm{Q}} - \tilde{\bm{Q}^T}\right)\odot \bm{F_m^r}\right]\bm{1}^T+\sum_{f=1}^{N_f}\sum_{k=1}^{N_{fp}} \bm{\chi}(\bm{\xi}_{f,k}^r)^T{W}_{f,k}\hat{\bm{n}}^r\cdot \bm{f}_m^{*,r}=\bm{0}^T.
\end{equation}
Here, $\bm{\chi}(\bm{\xi})$ are the basis functions written as polynomials in reference space, with $\bm{v}$ and $\bm{f}$ denoting both volume and facet reference points; while $\left(\tilde{\bm{Q}} - \tilde{\bm{Q}^T}\right)$ is the general hybridized skew-symmetric stiﬀness operator involving both volume and surface quadrature evaluations from~\cite{chan2019skew} and $\bm{F_m^r}$ is a matrix that contains the reference two-point fluxes. The boundary term is summed over $N_f$ faces at $N_{fp}$ surface quadrature points, while $\hat{\bm{n}}^r$ is the facet normal in reference space and $\bm{f}_m^{*,r}$ is the numerical flux.

The NSFR scheme, which serves as a foundation for the novel scheme proposed in this paper, is effectively a filtered DG scheme. The filtering of the DG residual is controlled by the flux reconstruction parameter, $c$, which is embedded in the modified mass matrix $\bm{K_m}$,
\begin{align}
    (\bm{K}_m)_{ij} &\approx \sum_{s,v,w } c_{(s,v,w)} \int_{ {\Omega}_r} J_m^\Omega \partial^{(s,v,w)} \chi_i(\bm{\xi}^r) 
    \partial^{(s,v,w)}\chi_j(\bm{\xi}^r) d {\Omega_r},
    \label{eq:Km}
\end{align}
\noindent where $\partial^{(s,v,w)}$ represents a multidimensional index spanning the $p$-th order broken Sobolev-space~\cite{sheshadri2016stability}. Setting $c=0$ allows us to recover a DG scheme, as the lifting operator is unfiltered. On the other hand, setting the $c$ parameter to the largest stable value, $c_+$, recovers an FR scheme that mitigates unwanted oscillations by filtering the lifting operator. The mitigation of the oscillations enhances the robustness of the scheme and allows for a higher CFL. However, the FR scheme that we recover has a greater magnitude of error than its DG counterpart, as demonstrated in \cite{lambert2023l2,castonguay2012high,castonguay2012new}. As such, the ideal scheme for shock-dominated problems would employ a DG scheme in smooth regions whilst adapting the lifting operator in regions dominated by strong shocks. This would allow us to retain the DG scheme's advantage of increased accuracy while also retaining the FR scheme's advantages of dampened oscillations and timestep efficiency. This is the motivation behind the novel adaptive flux reconstruction (AFR) method.

In this technical note, Sec.~\ref{sec: afr} outlines the preserved properties of the new scheme, Sec.~\ref{sec: methods} outlines key implementation details, including the shock sensor and timestepping method, and finally Sec.~\ref{sec: results} presents a suite of test cases that demonstrate the advantages of the new adaptive scheme over the DG and FR schemes. The implementation and analysis of this novel scheme are summarized in Sec.~\ref{sec: conclusion}.

\section{Adaptive Flux Reconstruction}\label{sec: afr}
The adaptive implementation presents a lot of desirable advantages, but it is crucial that the newly proposed scheme still retains the properties of the NSFR scheme that serves as a foundation. Two crucial properties of interest are conservation and nonlinear stability. In the following sections, the two main theorems of the NSFR method presented in \cite{cicchino2024weight} are revisited to show that the AFR scheme conserves these properties.

\subsection{Conservation}
The NSFR scheme possesses the crucial property of global and local conservation. Discretizations that approximate weak solutions to conservation laws must discretely satisfy conservation, and as such, the new scheme must retain this property. Prior to demonstrating that the adaptive flux reconstruction scheme satisfies local conservation, we first revisit several fundamental properties of ESFR schemes. 

\begin{remark}
    Following the derivations in \citet{zwanenburg_esfr}, ESFR can be expressed as a filtered DG scheme for general three-dimensional curvilinear coordinates on mixed element types with a modified norm; provided that the matrix $\bm{K}$ satisfy the following definition. 
\end{remark}

\begin{definition}\label{def:correction_function}
In two-dimensions, the vector correction functions $\bm{g}^{f,k}(\bm{\xi}^r)\in\mathbb{R}^{1\times d}$ associated with face $f$, facet cubature node $k$ in the reference element, are defined as the tensor product of the $p+1$ order one-dimensional correction functions ($\bm{\chi}_{p+1}$ stores a basis of order $p+1$), with the corresponding $p$-th order basis functions in the other reference directions,
\begin{equation}
\begin{split}
    \bm{g}^{f,k}(\bm{\xi}^r) = \Big[\Big(\bm{\chi}_{p+1}(\xi)\otimes \bm{\chi}({\eta}) \Big)\Big(\hat{\bm{g}}^{f,k}_1\Big)^T , \Big(\bm{\chi}({\xi})\otimes \bm{\chi}_{p+1}(\eta)\Big)\Big(\hat{\bm{g}}^{f,k}_2\Big)^T \Big] \\
    = [{g}^{f,k}_1(\bm{\xi}^r), {g}^{f,k}_2(\bm{\xi}^r)],
    \end{split}
\end{equation}
such that
\begin{equation}
    \bm{g}^{f,k}(\bm{\xi}_{f_i, k_j}^r) \cdot \hat{\mathbf{n}}^r_{f_i, k_j}
    =
    \begin{cases}
        1 \quad \text{if $f= f_i$ and $k= k_j$} \\
        0 \quad \text{otherwise}.
    \end{cases}
    \label{eq:BCs}
\end{equation}
\end{definition}
Following the derivations in Zwanenburg and Nadarajah [10] on the equivalence of energy stable flux reconstruction and discontinuous Galerkin methods, the modified mass matrix introduced previously is related to the correction function formulation via
\begin{equation}
    \nabla^r \cdot \bm{g}^{f,k}(\bm{\xi}^r)
    = \bm{\chi}(\bm{\xi^r})(\bm{M}+\bm{K})^{-1} \bm{\chi}(\bm{\xi}^r_{f,k})^T W_{f,k}.
    \label{eq:corr}
\end{equation}
Hence, NSFR schemes naturally give rise to a modified lifting operator on the surface integral, which can be expressed as
\begin{equation}
    (\bm{M}_m + \bm{K}_m)^{-1} \sum_{f=1}^{N_f} \sum_{k=1}^{N_{fp}} \bm{\chi}(\bm{\xi}_{f,k}^r)^T W_{f,k} \hat{\bm{n}}^r \cdot \bm{f}_m^{*,r}.
    \label{eq:mod_lift}
\end{equation}
This modified lifting operator will be associated with a conservative scheme provided that the conditions from the following theorem are satisfied. Note that in Eq.~\ref{eq:corr}, the mass matrix, $\bm{M}$ and FR  filter matrix $\bm{K}$ are evaluated on the reference elements as opposed to $\bm{M}_m$ and $\bm{K}_m$ as in Eq.~\ref{eq:Km} which includes the element Jacobian. 

\begin{theorem}
    The AFR discretization which uses the NSFR discretization in Eq.~\ref{eq: nsfrStrong} is locally conserving
\end{theorem}
\begin{proof}
    Starting with the NSFR discretization,
    \begin{equation}\label{eq: nsfrStrong}
        (\bm{M}_m+\bm{K}_m)\frac{d}{dt}\hat{\bm{u}}_m(t)^T + [\bm{\chi}(\bm{\xi_v^r})^T\bm{\chi}(\bm{\xi_f^r})^T]\left[\left(\tilde{\bm{Q}} - \tilde{\bm{Q}^T}\right)\odot \bm{F_m^r}\right]\bm{1}^T+\sum_{f=1}^{N_f}\sum_{k=1}^{N_{fp}} \bm{\chi}(\bm{\xi}_{f,k}^r)^T{W}_{f,k}\hat{\bm{n}}^r\cdot \bm{f}_m^{*,r}=\bm{0}^T 
    \end{equation}
    Given that $\bm{\chi}(\bm{\xi}_v^r)\hat{\bm{1}}^T = \bm{1}^T$, we left multiply by $\hat{\mathbf{1}}$. As outlined in \cite{cicchino2024weight}, $\left(\tilde{\bm{Q}} - \tilde{\bm{Q}^T}\right)\odot \bm{F_m^r}$ is skew-symmetric, its associated term vanishes and we are left with
        \begin{align}
        \hat{\mathbf{1}}(\bm{M}_m+\bm{K}_m)\frac{d}{dt}\hat{\bm{u}}_m(t)^T = -\sum_{f=1}^{N_f}\sum_{k=1}^{N_{fp}}{W}_{f,k}\hat{\bm{n}}^r\cdot \bm{f}_m^{*,r}
        \end{align}
    Using $\hat{\bm{1}}\bm{K}_m = 0$, this is further simplified to,
        \begin{equation}\label{eq: local_cons}
        \hat{\mathbf{1}}\bm{M}_m\frac{d}{dt}\hat{\bm{u}}_m(t)^T = -\sum_{f=1}^{N_f}\sum_{k=1}^{N_{fp}}{W}_{f,k}\hat{\bm{n}}^r\cdot \bm{f}_m^{*,r}
        \end{equation}
    which recovers local conservation. The adaptive implementation of the flux reconstruction parameter, $c$, contained within $\bm{K}_m$ does not impact local conservation, therefore the scheme is locally conservative.
\end{proof}

\begin{remark}
    Using Eq.~\ref{eq: local_cons}, the NSFR scheme is said to be globally conservative after summing across all of the elements and using a telescopic flux. This proof also follows for the AFR scheme.
\end{remark}

\subsection{Nonlinear Stability}
Another key property of NSFR is nonlinear stability. Provable nonlinear stability bounds the discrete approximation within a norm to the true solution and ensures that the discretization does not diverge. Provided that positivity of thermodynamic quantities are preserved, the discrete solution remains stable. While stability does not imply convergence, it is still an integral property to conserve in the AFR scheme to ensure that the scheme is stable and satisfies the second law of thermodynamics.

As per \citet{cicchino2025discretely}, Eq.~\ref{eq: nsfrStrong} satisfies the following formulation, 
\begin{equation}
    \sum_{n=1}^{n_{state}}\hat{\bm{v}}_n(\bm{M}_m + \bm{K}_m)\frac{d}{dt}\hat{u}_{m,n}(t)^T =  \sum_{n=1}^{n_{state}}\sum_{f=1}^{N_f}\sum_{k=1}^{N_{fp}}(\bm{\psi}_n(\bm{\xi_{fk}^r}- v_n(\bm{\xi_{fk}^r})\bm{f}_{m,n}^{*,r})\cdot\hat{\bm{n}^r}
\end{equation}
With the appropriate boundary conditions and an entropy conserving flux, discrete entropy conservation within the $(\bm{M}_m + \bm{K}_m)$-norm is proven.

\begin{remark}
    Akin to the NSFR discretization, the AFR discretization is discretely entropy conserving if the two-point flux satisfies the Tadmor shuffle condition. The adaptive implementation of the FR correction parameter does not affect the property of entropy stability as the discrete approximation is still bounded within a norm.
\end{remark}

\subsection{Other Considerations}\label{sec: afr_other}
It is also important to note that this implementation would require a positivity-preserving mechanism. In this implementation of the scheme, we use the positivity-preserving limiter introduced by \citet{zhang_ppl} with modifications from \cite{wang_ppl} and \cite{srinivasan2025investigation}. 

In addition to the positivity-preserving mechanism, we also require a suitable discontinuity sensor to be incorporated in order to effectively adapt the lifting operator according to the strength of the shock or discontinuity. In the proposed implementation of the scheme, the value of $c$ is scaled in the range of $[0,c_+]$. Any discontinuity sensor with an output in the range of $[0,1]$ can easily be incorporated into this adaptive scheme to scale the FR parameter. This allows for different sensors to be used depending on the case if required, thereby making it a flexible implementation. In this paper, the modal discontinuity sensor introduced by \citet{persson_modal} is used. The original implementation of the sensor is used, but the smooth function to determine the artificial dissipation is modified to suit the scaling of the $c$ parameter. The changes to the smooth function are discussed in Sec.~\ref{sec: sensor}.

%% file: sections/methodology.tex
\section{Methodology}\label{sec: methods}
This section lays out the details of the shock sensor and timestepping method chosen for the scheme's implementation.   

\subsection{Modal Sensor}\label{sec: sensor}
As described in Sec.~\ref{sec: afr_other}, the scheme requires a suitable shock sensor to determine the value of the flux reconstruction parameter. The modal sensor developed by \citet{persson_modal} uses the rate of decay of the expansion coefficients to determine the strength of the discontinuity and determines the appropriate amount of artificial viscosity to be added using the rate of decay and a smooth function. In this paper, a slight modification is made to the smooth function to determine the FR parameter instead of the artificial viscosity. Using $s_e = \log_{10}S_e$, where $S_e$ is the output of the sensor, and $s_0 = -4\log_{10}p$, we define the smooth function as,
\begin{equation}
    \varepsilon = \begin{cases}
        0 \qquad &\text{if} \: s_e < s_0 - \kappa, \\
        \frac{1}{2}\left(1 + \sin\frac{\pi(s_e - s_0)}{2}\right) \qquad &\text{if} \: s_0 - \kappa < s_e \leq s_0 + \kappa,\\
        1 \qquad  &\text{if} \: s_0 + 1 < s_e    \end{cases}
\end{equation}

We can see that $0 \leq \varepsilon \leq 1$, which can be used to scale the FR parameter, ie, $c_{sensor} = \varepsilon\cdot c_+$ where $0 \leq \varepsilon\cdot c_+ \leq c_+$. This yields a reliable value for the FR parameter based on the strength of the discontinuities and ensures that the solution uses FR only when required and applies DG elsewhere.

\subsection{Temporal Discretization}\label{sec: rk_method}
In this work, the solution is advanced in time using the strong stability preserving third-order accurate Runge-Kutta (SSPRK3) explicit time advancement method \cite{shu1988efficient}. The convective-based CFL condition used in this work is determined as follows:
\begin{align}
    \Delta t = CFL \: \frac{\Tilde{\Delta x}}{\lambda_{max}}, \quad \Tilde{\Delta x} = \frac{x_{max}-x_{min}}{(DOF)^{1/dim}}, \quad \lambda_{max} = max(|v|+c),
\end{align}
where $\Tilde{\Delta x}$ is the approximate grid spacing and $\lambda_{max}$ is the maximum wavespeed of the initial condition with speed of sound $c = \sqrt{\gamma R T}$.

%% file: sections/results.tex
\section{Numerical Results}\label{sec: results}
\input{sections/results/gaussian-pulse}
\input{sections/results/leblanc-shock-tube}
\input{sections/results/shock_diffraction}
\input{sections/results/double-mach-reflection}

%% file: sections/results/gaussian-pulse.tex
\subsection{Gaussian Pulse Test}
\label{sec: gaussian-pulse-test}
The first test case is an order of accuracy test that involves a smooth convecting density pulse as seen in \citet{dzanic_gaussian_pulse}. With a smooth problem such as this, it is expected that both the DG and FR schemes will converge at an order of $p+1$, but DG will have a lower magnitude of error than FR. While the density pulse is smooth, the sharp gradients at the pulse still activate the shock sensor, resulting in the adaptation of the lifting operator in certain regions of the domain. This makes it the ideal test case to demonstrate that the new adaptive scheme still retains the desired $p+1$ order of accuracy. The problem is solved on a domain of size $[-0.5,0.5]^2$ with periodic boundary conditions and initial condition,
\begin{equation*}
    \rho_0(x,y) = 0.01 + exp(-\sigma (x^2+y^2)),\ u_0(x,y) = 1,\ v_0(x,y)=1,\ p_0(x,y)=1,
\end{equation*}
where $\sigma = 500$ is the strength of the pulse. The final time is $t=1.0s$, which results in the exact solution being equal to the initial condition. 

This convergence test is run for the DG and FR schemes in addition to the adaptive scheme to demonstrate the accuracy of the solution compared to the original schemes. The plots provided in Figure~\ref {fig:gp_ooa_plots} show the $L_1$, $L_2$ and $L_\infty$ errors for polynomial degrees $2$ and $3$ with grid size starting at $8\times8$ and successively refined to $256\times256$.

The reference line shown in the plots provides the expected error if the error diminishes at a rate of $p+1$ starting from the error for the DG scheme at the second coarsest grid. We see from the plots that the original DG and FR schemes converge with an order of $p+1$, and the FR scheme has a greater magnitude of error than the DG scheme. The error for the adaptive scheme lies within the error of the two original schemes, which is expected since the scheme incorporates both DG and FR in the solution. It is noted that as the grid is refined, the error for the new scheme is identical to DG since the shock sensor is not turned on for the finer grids. This test demonstrates that the new scheme preserves a high order of accuracy and that the adaptation of the lifting operator in select regions does not contribute to significant error. This slight increase in error is negligible when considering the robustness and timestep advantages that the new scheme inherits from FR. This increase in robustness and timestep will be showcased in the following test cases.
\begin{figure}[ht]
    \centering
    \includegraphics[width=0.35\textwidth]{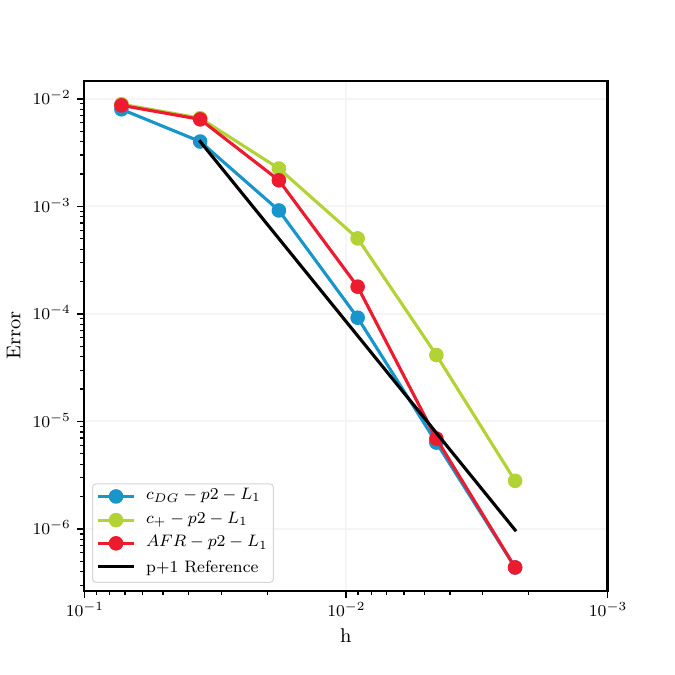}\includegraphics[width=0.35\textwidth]{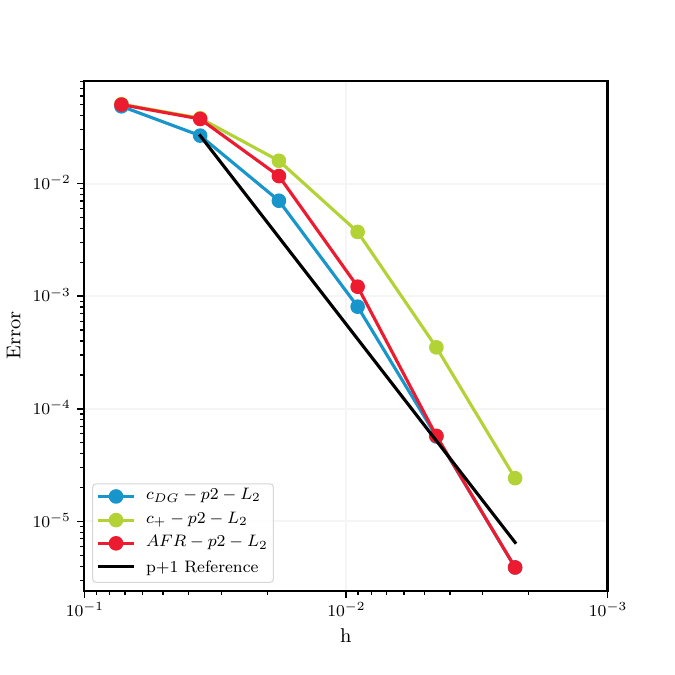}\includegraphics[width=0.35\textwidth]{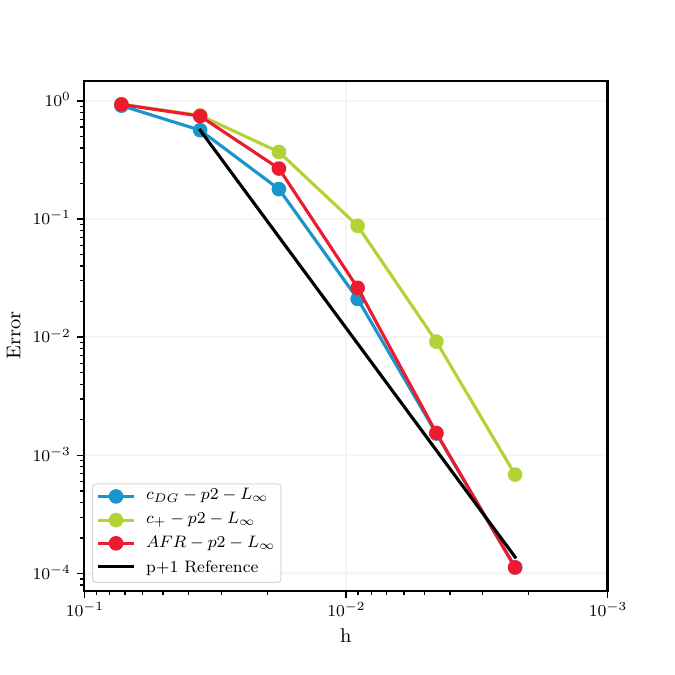}\\
   \includegraphics[width=0.35\textwidth]{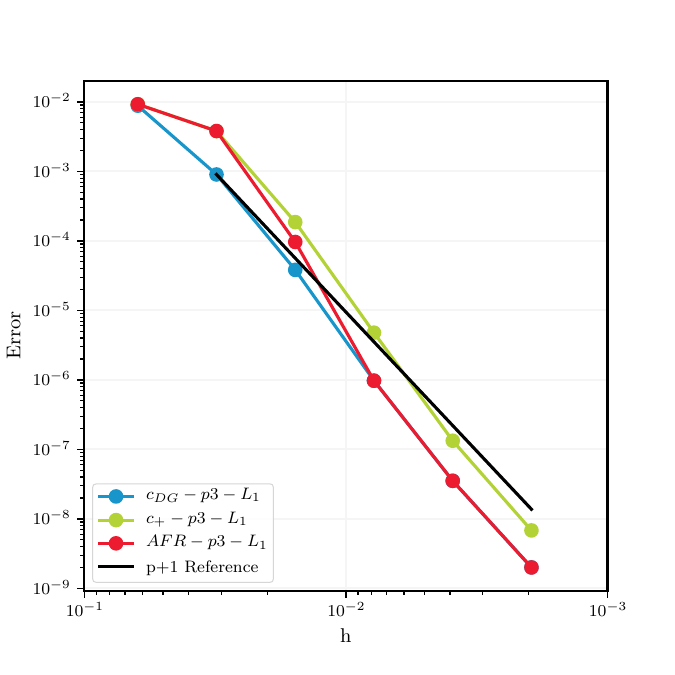}\includegraphics[width=0.35\textwidth]{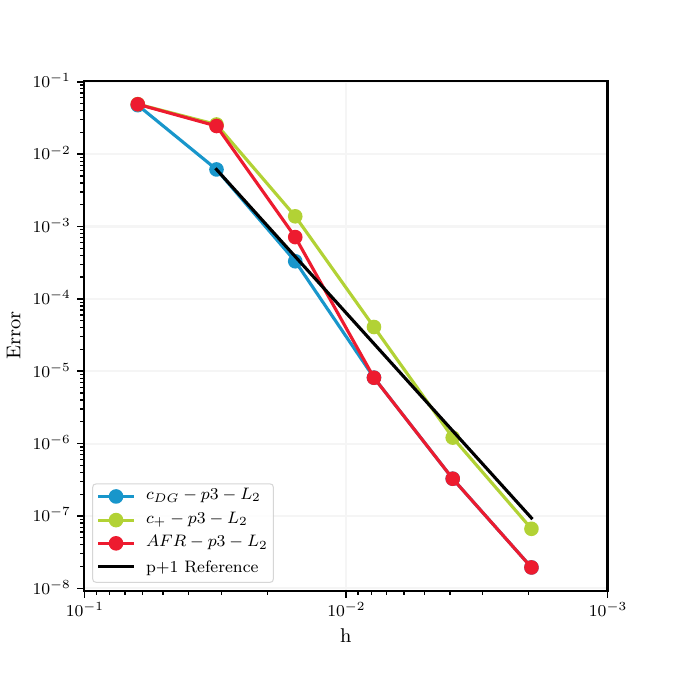}\includegraphics[width=0.35\textwidth]{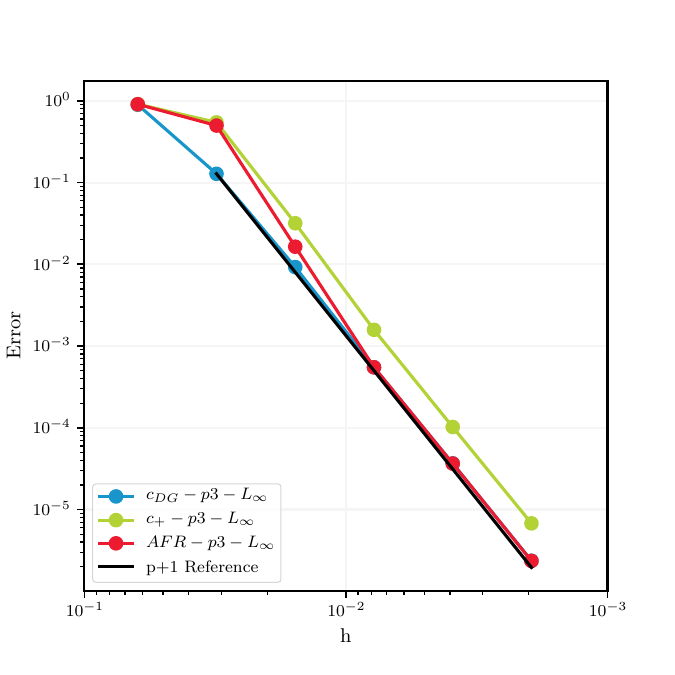}
    \caption{\textit{[Gaussian Pulse]} Plot of orders of convergence for the test results using DG, adaptive and FR schemes. Top row (p2):  $L_1$ (Left), $L_2$  (Middle) and $L_\infty$  (Right). Bottom row (p3):  $L_1$ (Left), $L_2$  (Middle) and $L_\infty$  (Right).}
    \label{fig:gp_ooa_plots}
\end{figure}

%% file: sections/results/leblanc-shock-tube.tex
\subsection{Leblanc Shock Tube}
\label{sec: leblanc-shock-tube}
The modified version of the Leblanc Shock Tube case, as seen in \citet{zhang_ppl}, is used to further demonstrate the advantages of the adaptive scheme in cases involving strong shocks. The computational domain is $[-10.0, 10.0]$. The problem is initialized as,
\begin{equation}
    (\rho, w, p) = \begin{cases}
        (2,0,10^{9}) \qquad &\text{if} \: x < 0,\\
        (0.001,0,1) \qquad &\text{if} \: x \geq 0.
    \end{cases}
\end{equation}
The final time for this case is $t = 1\times 10^{-4}$. First, a comparison of maximum $CFL$ is conducted for polynomial degres $3$, $4$, and $5$. The degrees of freedom are kept at 1920 for all polynomial orders. The following table, Table \ref{tab:leblanc_cfl}, outlines the maximum $CFL$ for all three schemes for the polynomial orders $3,4,$ and $5$. The DG scheme has a $CFL$ number that is lower than that of the other schemes. The $CFL$ is approximately $5$ times and $40$ times less than the other schemes for $p=3$ and $p=4$, respectively. Additionally, the test fails for DG for $p=5$ despite using a low $CFL$ of $0.005$ while the other two schemes have a significantly higher $CFL$. The $p=5$ case is of special interest since we are not able to obtain a solution using the DG scheme, but with minimal adaptation, the scheme is able to obtain a solution with a reasonably high $CFL$. The adaptive scheme primarily uses DG across the solution except at the shocks and discontinuities, as demonstrated by the blast wave case, so we are able to achieve this significant improvement in robustness compared to the DG scheme by simply incorporating FR at the extremities of the solution. This demonstrates a clear advantage for the adaptive scheme over the classical DG scheme and highlights the need for the adaptive scheme, as it allows for a greater timestep to be taken while maintaining a lower level of error compared to the FR scheme.

\begin{table}[ht]
\centering
\caption{Maximum $CFL$ Number for each Leblanc Shock Tube Test Configuration}
\label{tab:leblanc_cfl}
\begin{tabular}{@{}cccc@{}}
\toprule
\multicolumn{1}{c}{Scheme} & \multicolumn{1}{c}{p3} & \multicolumn{1}{c}{p4} & \multicolumn{1}{c}{p5}\\ \midrule
DG    & $0.1$    & $0.005$   &\\
Adaptive       & $0.29$    & $0.21$    & $0.15$\\
FR    & $0.3$     & $0.23$    & $0.19$\\
\bottomrule
\end{tabular}
\end{table}

In addition to the $CFL$ tests, a grid refinement test is also conducted. Fine mesh resolutions are often required for the Leblanc Shock Tube, as it is a challenging problem to simulate. Added dissipation at the contact discontinuity can lead to smearing, which results in an overprediction of specific internal energy and consequently overpredicts the shock speed \cite{liu2009high}. This grid refinement study compares the performance of the adaptive scheme to the original DG and FR schemes as well as the exact solution. The density plots for the $p=3$ solution are presented in Fig.~\ref{fig:leblanc_p3} for meshes of size $N=900$ and $N=3600$. The adaptive and FR schemes were run with the $CFL$s reported in Table \ref{tab:leblanc_cfl}. The DG scheme on the other hand required the $CFL$ to be dropped to $0.06$ for the $N=900$ grid and could not be run for the $N=3600$ grid for a $CFL$ as low as $0.01$. 

In the first density plot, the final shock location of all three schemes is very close to the exact solution, just slightly offset. The offset of the final shock location also results in the region to the right of the rarefaction wave not reaching the correct magnitude. In addition to this, it is observed that the DG solution is extremely oscillatory in the wake of the shock and contact discontinuity due to the unmitigated spurious oscillations. The adaptive solution does not contain such oscillations despite the minimal adaptation performed, which is reflected in the FR Parameter plot below the density plot. This plot corresponds to the strength of the shock at that location and indicate where the lifting operator is adapted across the solution for the new scheme. The maximum possible value for the FR parameter in the $p=3$ test configuration is $c = 2.87e-5$. The highest value of $c$ used for the adaptive scheme is approximately $2.64e-5$ for $N=900$ at the primary shock, and the rest of the domain either requires little to no adaptation.

In the finer grid, we see that the final shock location for adaptive and FR match almost perfectly with the exact solution, and the magnitude of the solution to the right of the rarefaction wave is correct. Similar to the first plot, the FR parameter plot below the density plots indicates where the lifting operator is adapted across the solution for the new scheme. The highest value of $c$ used for the adaptive scheme is approximately $2.73e-5$ for $N=3600$ despite the maximum allowed value being $c = 2.87e-5$. This value is only employed at the shock, and the rest of the domain employs much lower values for the FR parameter (a majority of the domain employs a value of $0$). This case demonstrates the novel scheme's advantages as it is able to successfully simulate the Leblanc Shock Tube case with a $CFL$ close to the FR scheme, and it follows the exact solution closely with minimal adaptation, while the original DG scheme cannot run this test case in the given configuration. The novel scheme retains the advantages of FR while primarily employing DG and produces exemplary results for even extreme benchmark problems such as the Leblanc Shock Tube case.

\begin{figure}[h]
    \includegraphics[width=0.45\textwidth]{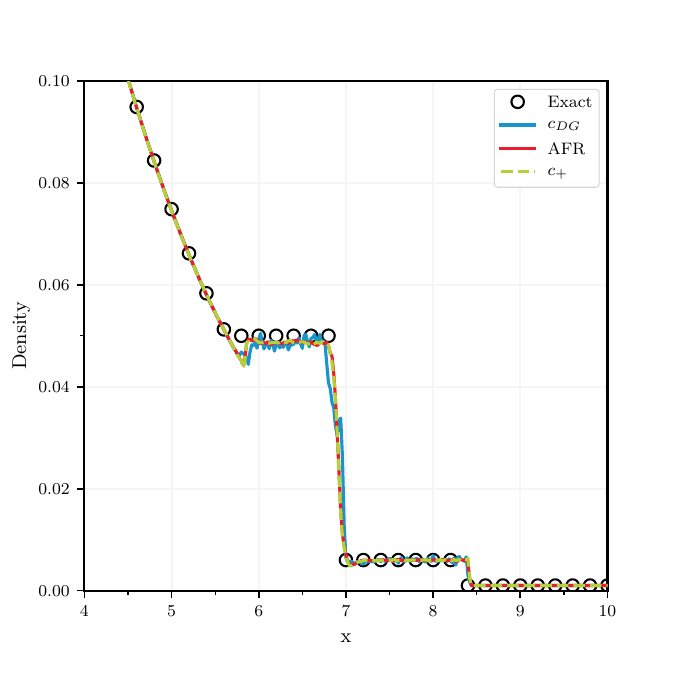}\includegraphics[width=0.45\textwidth]{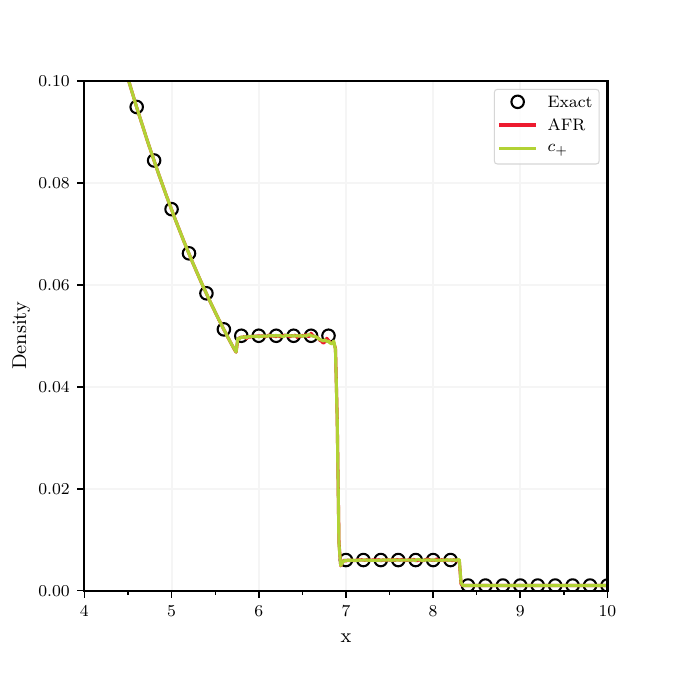}\\
    \includegraphics[width=0.45\textwidth]{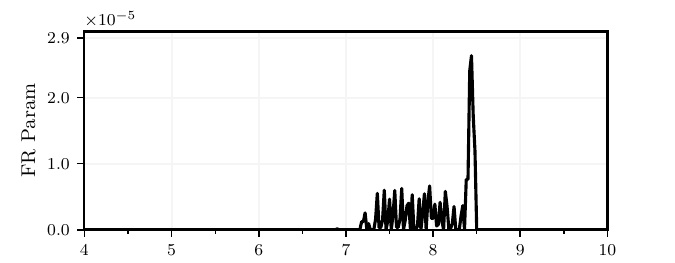}\includegraphics[width=0.45\textwidth]{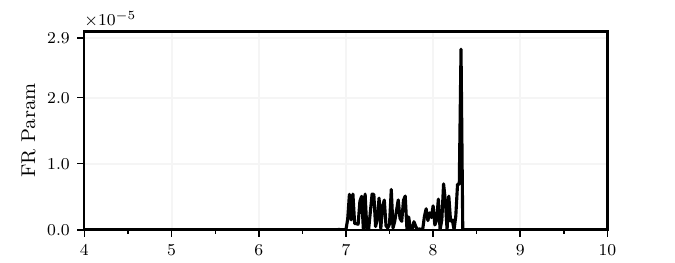}
    \caption{\textit{[Leblanc Shock Tube]} Plots of density and pressure within $15 \leq x \leq 18.75$ for $p=3$, NSFR-Ra, Roe dissipation for the $c_{DG}$, $c_+$ and AFR schemes with the plots underneath indicating the value of $c$ used for the AFR scheme}
    \label{fig:leblanc_p3}
\end{figure}

%% file: sections/results/shock_diffraction.tex
\subsection{Shock Diffraction}
\label{sec: shock-diffraction}
The shock diffraction case is a 2D problem wherein a Mach 5.09 shock propagates rightwards and diffracts at a $90\text{\textdegree}$ backward-facing corner. The solution is expected to have several key features. These features include the incident shock, the diffracted shock, the reflected expansion wave, and the vortex roll-up as highlighted in \cite{hillier1991computation}. This case is a good test of the schemes as the density below the $90\text{\textdegree}$ corner is very close to zero, and the oscillations in the wake of the shock result in nonphysical values occurring easily in the DG scheme if the $CFL$ is not substantially low. The adaptive scheme and the FR scheme both mitigate the oscillations in the wake of the shock, thereby allowing for a higher $CFL$ and a solution that better reflects those seen in literature such as \cite{zhang_ppl,xu2022third}. The final time used is $T=2.3$. The computational domain is $[0,1]\times[6,11]\cup[1,13]\times[0,11]$. The initial condition is
\begin{equation}
    (\rho, u, v, p) = \begin{cases}
        (\rho_{\text{left}}, u_{\text{left}}, v_{\text{left}}, p_{\text{left}}) \qquad &\text{if} \: x < 0.5\\
        (1.4,0,0,1) \qquad &\text{if} \: x > 0.5\\
    \end{cases}\\
\end{equation}
Where $\rho_{\text{left}} = 7.041132906907898$, $u_{\text{left}} = 4.07794695481336$, $v_{\text{left}} = 0$, $p_{\text{left}} = 30.05945$. The conditions for the left boundary above the backwards-facing corner are inflow (${x=0,6\leq y\leq 11}$) and wall boundary below the corner (${x=1,0\leq y\leq6}$). The bottom boundary is a wall below the corner (${0\leq x\leq1,y=6}$) and outflow above the corner (${y=0,1\leq x\leq13}$). The right boundary, ${x=13,0\leq y\leq11}$, is outflow and the top boundary, ${y=11,0\leq x\leq13}$, is a wall boundary. The AFR scheme is first run with its maximum $CFL$ value of $0.61$. The results are shown below in Fig.~\ref{fig:sd_afr} for the density and FR parameter.
\begin{figure}[h]
    \centering
    \includegraphics[width=0.47\linewidth]{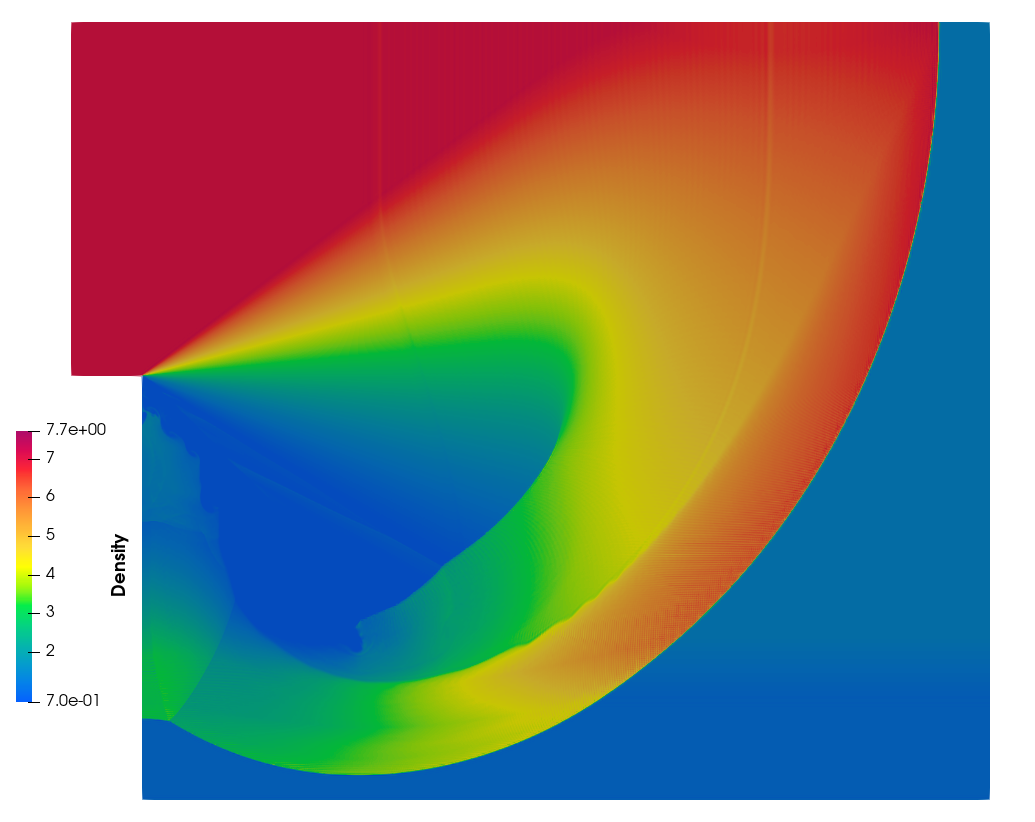}
    \includegraphics[width=0.47\linewidth]{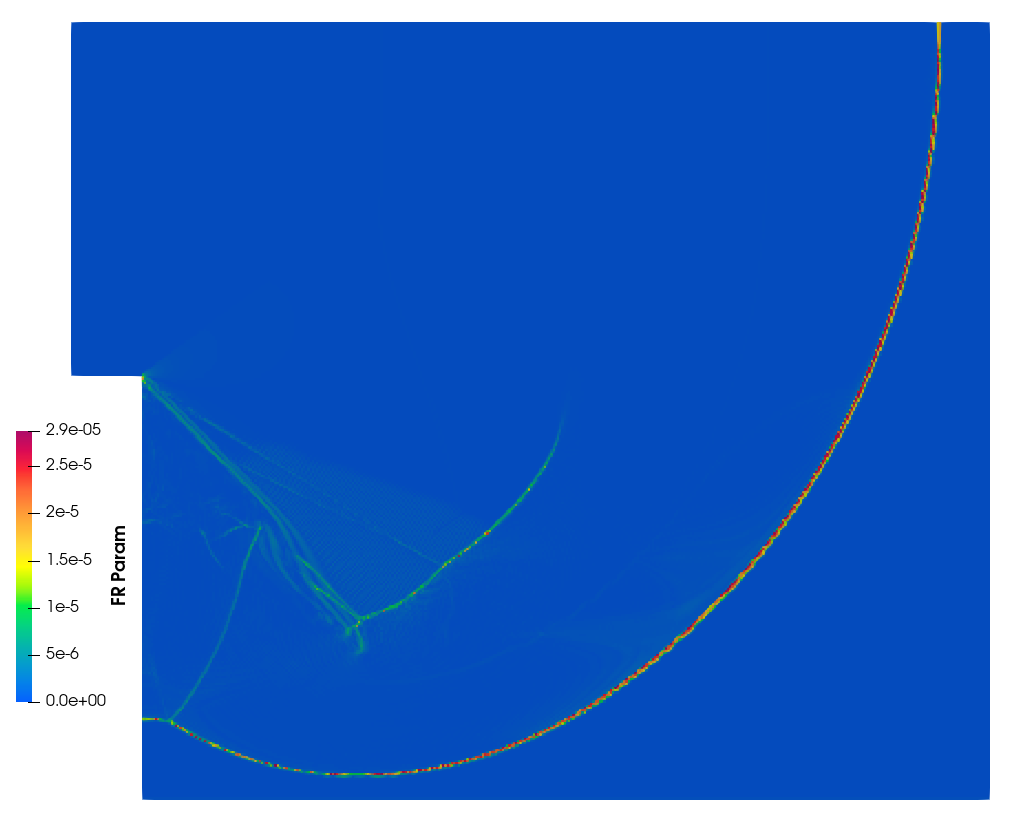}
    \caption{\textit{[Shock Diffraction]} P3 solution at $t=2.3s$ using Adaptive scheme - Density (left) and FR Parameter (right)}
    \label{fig:sd_afr}
\end{figure}

The plot of the density includes all the expected solution features, and the plot of the FR parameter shows that the maximum possible value for the parameter ($c=2.87e-5$) is only reached at the shock itself, while primarily DG is used everywhere else. This limited adaptation of the lifting operator allows for the $CFL$ to be increased from the maximum possible value of $0.45$ for DG, to a value of $0.61$, which is the maximum $CFL$ for the FR scheme before failure. Below, in Table~\ref{tab:sd_cfl}, the maximum $CFL$ numbers for each test configuration are shown, demonstrating the advantages of the adaptive scheme.
\begin{table}[ht]
\centering
\caption{Maximum $CFL$ Number for each Shock Diffraction Test Configuration}
\label{tab:sd_cfl}
\begin{tabular}{@{}cc@{}}
\toprule
\multicolumn{1}{c}{Scheme} & \multicolumn{1}{c}{p3}\\ \midrule
DG    & $0.45$ \\
Adaptive       & $0.61$  \\
FR    & $0.61$ \\
\bottomrule
\end{tabular}
\end{table}
To further demonstrate the advantages of the adaptive scheme, this test case is run with all three schemes for a $CFL$ of $0.45$. Plots are provided in Fig.~\ref{fig:sd_comparison} for the density solution of all schemes along the line $(1,0)\rightarrow(13,11)$. The tests are run for the maximum $CFL$ for each scheme respectively. Plots are provided for the FR parameter below the density plots to show how much the lifting operator is adapted at those points.

\begin{figure}[h]
    \centering
    \includegraphics[width=0.43\linewidth]{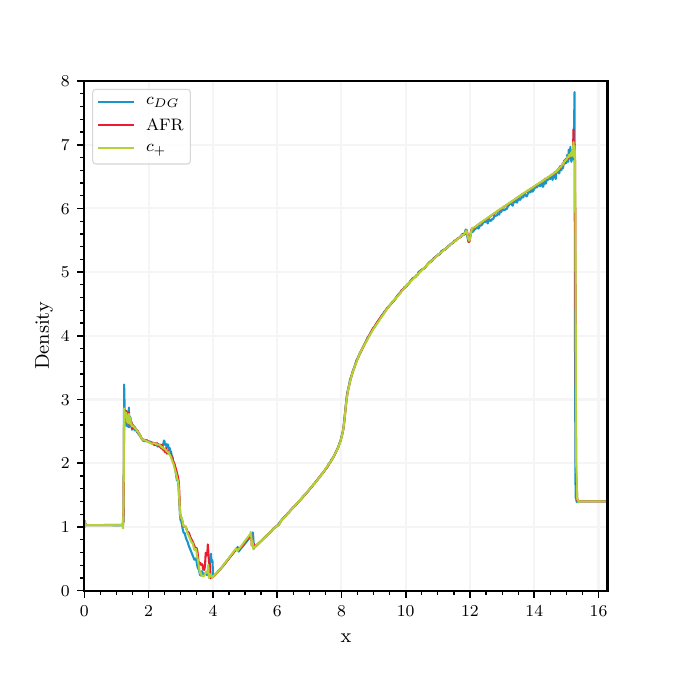}
    \includegraphics[width=0.43\linewidth]{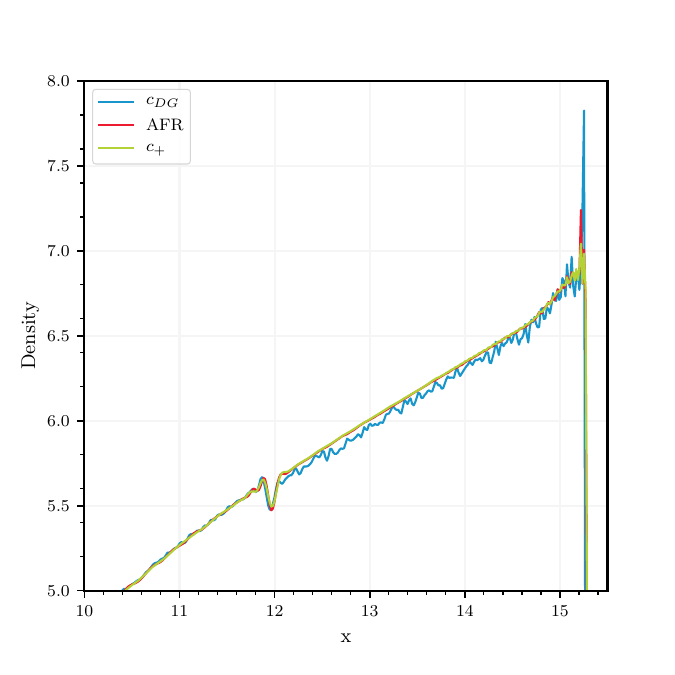}\\
    \includegraphics[width=0.43\linewidth]{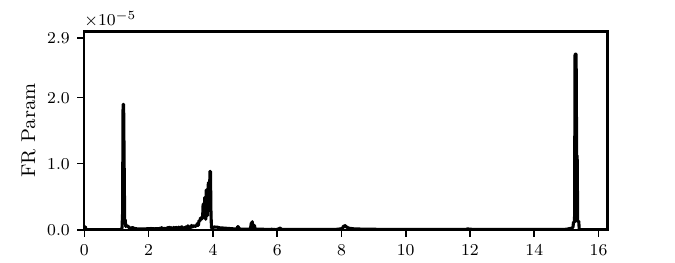}
    \includegraphics[width=0.43\linewidth]{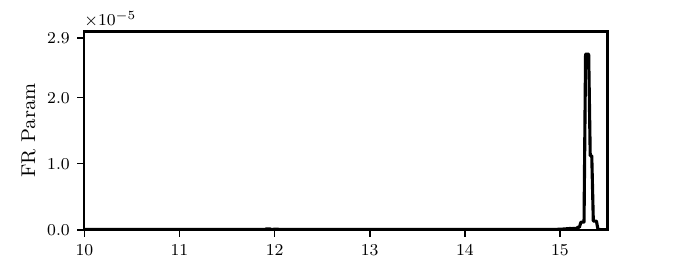}
    \caption{\textit{[Shock Diffraction]} Plot of density across $(1,0)\rightarrow(13,11)$ at $t=2.3s$ using for all three schemes (left) with a closer look at the wake of the shock (right)}
    \label{fig:sd_comparison}
\end{figure}

The plots effectively demonstrate that the DG scheme is very oscillatory and also has greater overshoot at the shock. The other two schemes achieve a solution that contains fewer oscillations and overshoot, especially in the wake of the shock. The FR parameter plots below indicate that the adaptation is primarily performed at the shocks, and DG is used everywhere else. Thus, through this minimal adaptation, we are able to see a great improvement in the mitigation of overshoots and oscillations, further demonstrating the advantage of the novel scheme.

%% file: sections/results/double-mach-reflection.tex
\subsection{Double Mach Reflection}\label{sec: dmr}
The Double Mach Reflection (DMR) problem introduced by \citet{woodward_blast_dmr} involves a Mach 10 shock colliding with a reflecting wall. This paper uses the rectangular domain setup of the problem. The domain is $[0,4]\times[0,1]$ with the initial condition,
\begin{equation}
    (\rho, u, v, p) = \begin{cases}
        (8,\frac{33\sqrt{3}}{8},-\frac{33}{8},116.5) \qquad &\text{if} \: y > \sqrt{3}(x-\frac{1}{6})\\
        (1.4,0,0,1) \qquad &\text{if} \: y < \sqrt{3}(x-\frac{1}{6})
    \end{cases}
\end{equation}
The left boundary is inflow, and the right boundary is outflow. The bottom boundary corresponding to $0\leq x<\frac{1}{6}$ has the post-shock condition, while the rest of the bottom boundary is a wall boundary. To minimize implementation effort, the domain is extended in the $y$-direction such that the top boundary doesn't affect the domain of the test case. Thus, the computational domain is $[0,4]\times[0,3]$. The final time is $t=0.2s$. As outlined in \citet{vevek2019alternative}, the solution at the final time is a self-similar shock structure with two triple points and two slip lines. The reflected shock interacts with the primary slip line and generates KHI along the slip line. A jet is also present near the wall. To capture all the complex features of the solution, especially the secondary slip line, we require a high-order scheme with robust shock-capturing to mitigate the noise in the solution. As such, this case demonstrates the need for an adaptive scheme over the classical DG scheme.

\begin{figure}[h]
    \centering
    \includegraphics[width=0.85\textwidth,trim={0cm 2cm 0 3.5cm},clip]{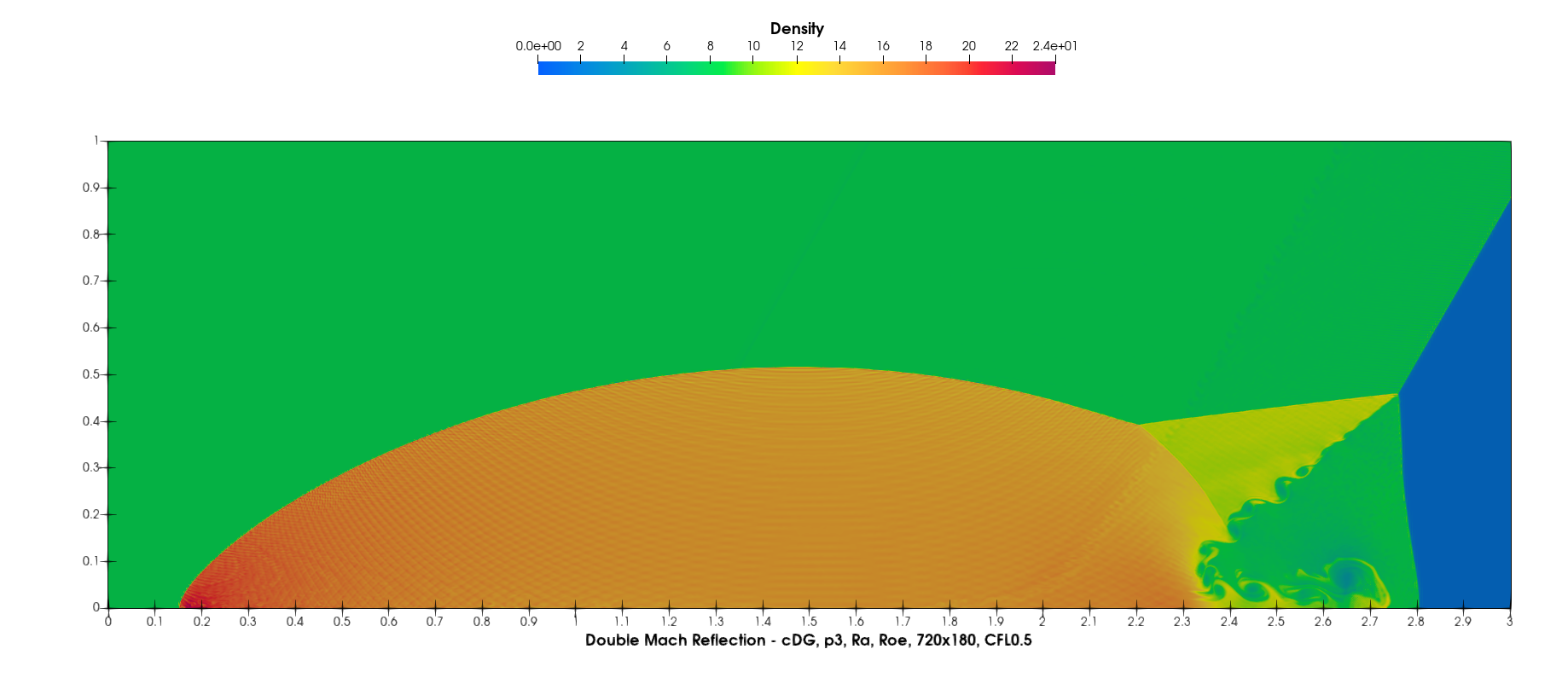}\\
    \includegraphics[width=0.85\textwidth,trim={0cm 2cm 0 3.5cm},clip]{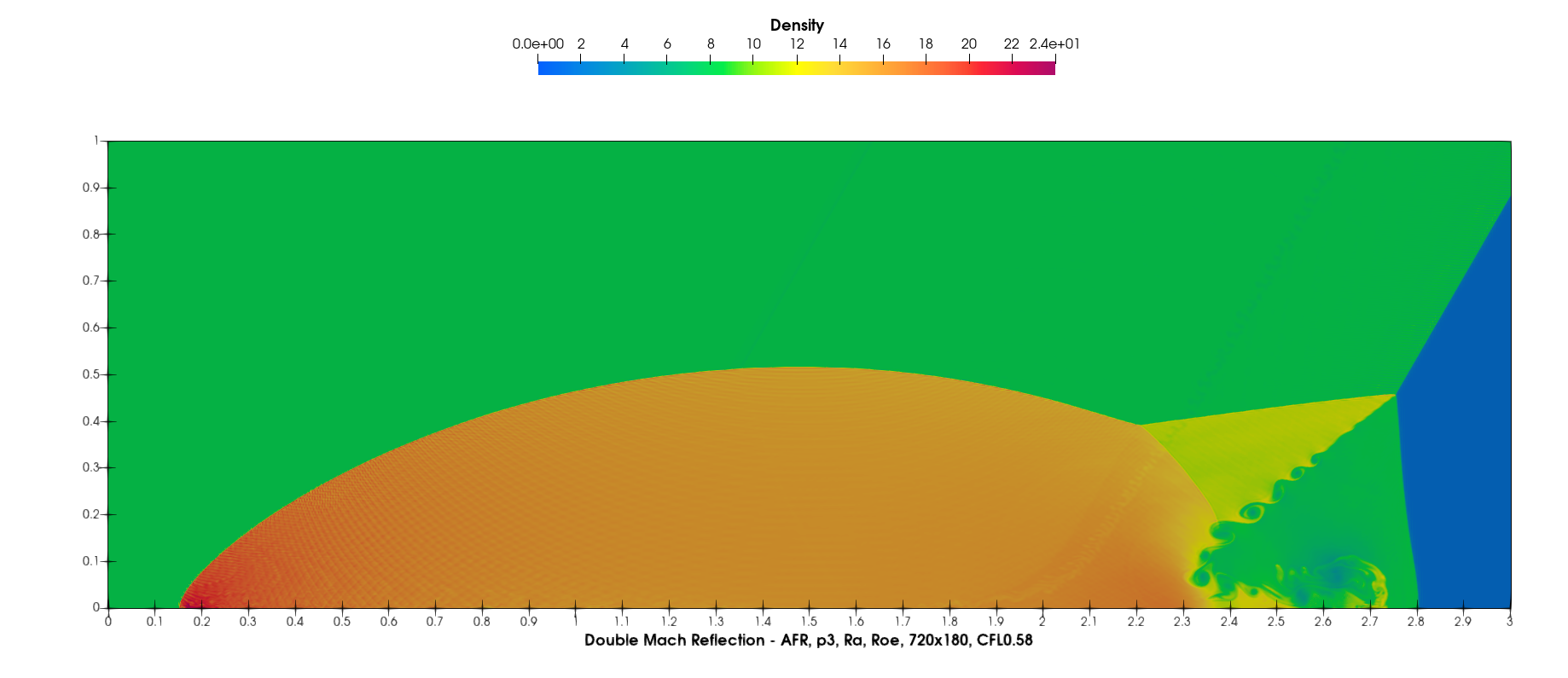}\\
    \includegraphics[width=0.85\textwidth,trim={0cm 2cm 0 3.5cm},clip]{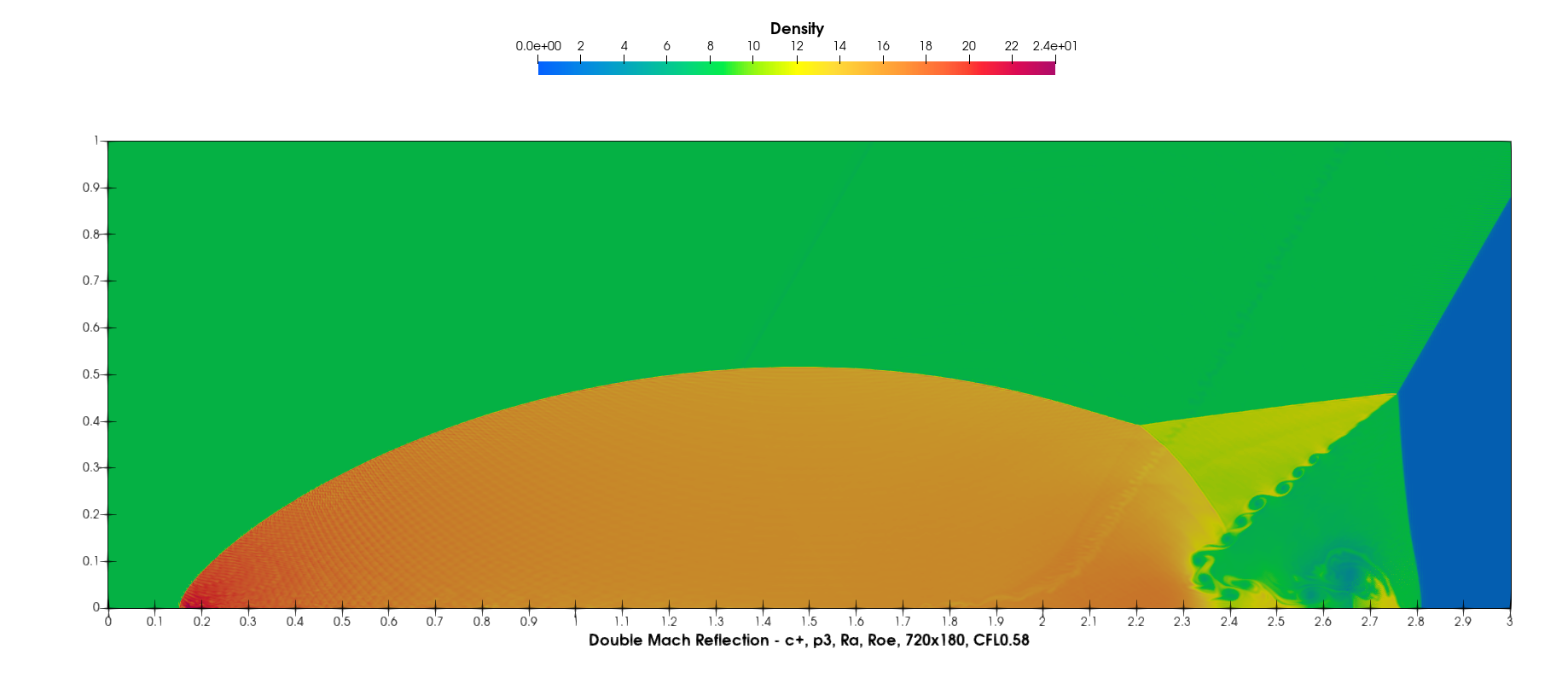}\\
    \caption{\textit{[DMR]} P3 Density solution at $t=0.2s$, using schemes - DG (top), Adaptive (middle), FR (bottom)}
    \label{fig:DMR_density}
\end{figure}

In Fig.~\ref{fig:DMR_density}, the solutions obtained using the different schemes are shown. 
All three solutions exhibit the expected features; however, the DG solution suffers from spurious oscillations, introducing significant noise that obscures the secondary slip line.
For the adaptive and FR schemes, the adaptation of the lifting operators leads to mitigation of the noise, and the secondary slip line is clearly present in the solution. In Fig.~\ref{fig:DMR_sensor}, the value of the FR parameter used in the adaptive scheme is shown. The maximum value for $p=3$ is $c=2.87e-5$, and this value is only reached in the region where the shock interacts with the wall. The rest of the solution, where the solution is adapted to FR, the value of $c$ is approximately $1/3$ of the value which is used across the entire domain for the FR scheme. The figure also shows that most of the solution is primarily DG. This demonstrates the advantages of using the adaptive scheme for strong shock problems that require high resolution, as we are able to obtain a solution that closely reflects the literature and contains minimal noise just by adapting the scheme to FR at the shocks.

\begin{figure}
    \centering
    \includegraphics[width=0.85\textwidth,trim={0cm 2cm 0 0.75cm},clip]{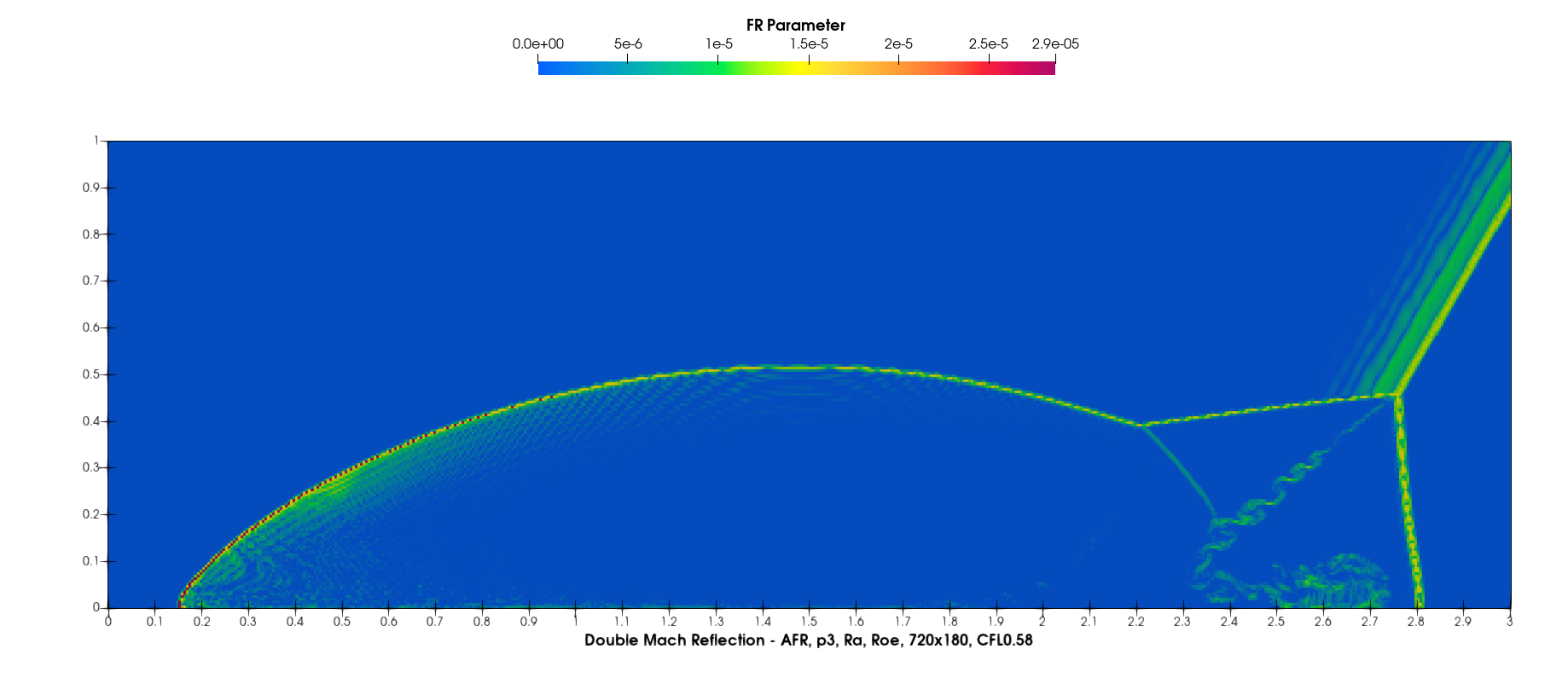}\\
    \caption{\textit{[DMR]} Plot of FR parameter at $t=0.2s$ for the adaptive scheme}
    \label{fig:DMR_sensor}
\end{figure}

%% file: sections/conclusion.tex
\section{Conclusion}\label{sec: conclusion}
In this technical note, a novel scheme has been presented to overcome the phenomenon of unwanted high-frequency oscillations commonly present in strong shock problems while also retaining a high level of accuracy. The new scheme is an adaptive implementation of the NSFR scheme that uses the modal shock sensor of \citet{persson_modal} to incorporate FR into the DG scheme when shocks and discontinuities are detected. This is achieved by adjusting the local lifting operator through the integration of the FR matrix, which is scaled according to the sensor. This scheme retains all the critical properties of the NSFR scheme while also retaining the advantages of the DG and FR schemes, which it incorporates. The new scheme is shown to possess a lower error compared to the FR scheme in Sec.~\ref{sec: gaussian-pulse-test} and it also has a significant advantage in robustness and timestep efficiency compared to the DG sche,me which is demonstrated in Sec.\ref{sec: leblanc-shock-tube}, \ref{sec: shock-diffraction} and \ref{sec: dmr}. This novel scheme presents important advantages over both the original schemes that it incorporate,s thereby making it a desirable implementation. 

%% file: sections/acknowledgements.tex
\section*{Acknowledgments}
\indent Sai Shruthi Srinivasan thanks the Vadasz Family Foundation and McGill Engineering Doctoral Award (MEDA). The authors would also like to thank Mathias Dufresne-Piché for helpful discussions throughout the course of this work. 